\documentclass[11pt,reqno]{amsart}
\usepackage{amsmath,amsfonts,amsthm,amssymb,color,bbm}
\usepackage[usenames,dvipsnames]{xcolor}
\usepackage[T1]{fontenc}
\usepackage{enumerate}
\usepackage{mathrsfs}
\usepackage{hyperref}
\hypersetup{
     colorlinks   = true,
     citecolor    = blue,
     linkcolor    = blue
}

\usepackage{pdfsync}{\tiny }

\usepackage[left=1in, right=1in, top=1.1in,bottom=1.1in]{geometry}
\numberwithin{equation}{section}

\def\R{\mathbb{R}}
\def\P{\mathbb{P}}
\def\C{\mathbb{C}}

\def\wh{\widehat}
\def\N{\mathbb{N}}

\def\H{\mathfrak{H}}

 \def\E{\mathbb{E}}

\def\to{\rightarrow}

\def\H{\mathfrak{H}}

\newtheorem{thm}{Theorem}[section]
\newtheorem{lemma}[thm]{Lemma}
\newtheorem{cor}[thm]{Corollary}

\newtheorem{obs}[thm]{Remark}

\def\1{\mathbbm{I} }

\begin{document}

\title{Almost sure convergence on chaoses}

 \date{\today}

\author{Guillaume Poly \and Guangqu Zheng}

\address{Guillaume Poly:  Institut de recherche math\'ematique de Rennes (UMR CNRS 6625), Universit\'e de Rennes 1, 
B\^atiment 22-23, 263, avenue du G\'en\'eral Leclerc, 35042 Rennes CEDEX, France}
\email{guillaume.poly@univ-rennes1.fr}

\address{Guangqu Zheng: University of Kansas, Mathematics department, Snow Hall, 1460 Jayhawk blvd, Lawrence, KS 66045-7594, United States}
\email{gzheng90@ku.edu}

\subjclass[2010]{Primary: 60F99; Secondary: 60H05.}

\keywords{exchangeable pairs, Ornstein-Uhlenbeck semigroup, Wiener chaos, homogeneous sums.}

\maketitle

\begin{abstract}

We present several new phenomena about almost sure convergence  on homogeneous chaoses that include Gaussian Wiener chaos and homogeneous sums in independent random variables. Concretely, we establish the fact that almost sure convergence on a fixed finite sum of chaoses forces the almost sure convergence of each chaotic component. Our strategy uses  ``{\it extra randomness}'' and a simple conditioning argument. These ideas are close to the spirit of \emph{Stein's method of exchangeable pairs}. Some natural questions are left open in this note. 

\end{abstract}

\section{Introduction}

Before presenting our results, let us discuss briefly the following simple problem. Assume that, on a suitable probability space, we have two sequences of random variables $\{X_n\}_{n\ge 1}$ and $\{Y_n\}_{n\ge 1}$ which fulfill the following conditions:

\begin{itemize}
\item[a)] for every $n\geq 1$ we have $\E(X_n)=\E(Y_n)=0$,
\item[b)] for some $M>0$ and every $n\geq 1$ we assume that $\max(|X_n|,|Y_n|)\leq M$ almost surely,
\item[c)] for every $n\geq 1$, $\E(X_n Y_n)=0$,
\item[d)] $X_n+Y_n\xrightarrow[n\to\infty]{\text{a.s.}}~0.$ (a.s. is short for ``almost surely''.)
\end{itemize}

\begin{center}
{\it Do we necessarily have that $X_n\xrightarrow[n\to\infty]{\text{a.s.}}~0$ and $Y_n\xrightarrow[n\to\infty]{\text{a.s.}}~0$?}
\end{center}

Using dominated convergence and the orthogonality assumption c), we get $\E(X_n^2)+\E(Y_n^2)\to 0$ which is not far from   a positive answer to the previous question, up to extracting a subsequence. However, without additional structure assumptions, the answer is {\it totally negative}, as indicated by the following counterexample.

The starting point is to consider an uniformly bounded sequence of $1$-periodic functions that are centered and converge to zero in $L^2([0,1])$ but not almost surely for the Lebesgue measure $\lambda$ on $[0,1]$. A standard way to build this object is to consider for every $p\ge 0$ and every integer $n \in [2^p,2^{p+1}-1]$, the indicator function $f_n$ of the set $\left[ (n-2^p)2^{-p}, (n-2^p+1)2^{-p}\right]$. The convergence to zero of the $L^2([0,1])$-norm is straightforward while one may observe that for every non dyadic $x \in [0,1]$ and every $A>0$ one can find $n\ge A$ such that $f_n(x)=1$, which disproves the almost sure convergence. We shall denote by $\{f_n\}_{n\ge 1}$ the resulting sequence of functions and by $\tilde{f}_n$ the centered functions $f_n-\int_0^1 f_n(x) dx$. Clearly the previous observations are also valid for $\tilde{f}_n$, since $\int_0^1 f_n(x) dx \to 0.$

Next, given 
$
(e_p(x),f_p(x))_{p\ge 0}:=\big(\sqrt{2}\cos(2 p \pi x),\sqrt{2}\sin(2 p \pi x) \big)_{p\ge 0}
$ an orthonormal basis of $L^2([0,1])$,  we can expand $\tilde{f}_n$ in the $L^2([0,1])$ as follows:
$$
\tilde{f}_n=\sum_{p=1}^\infty \alpha_{p,n} e_p +\beta_{p,n} f_p \, .
$$

Besides, set $\epsilon_n=\sum_{p=2}^\infty \alpha_{p,n}^2+\sum_{p=1}^\infty \beta_{p,n}^2$ and $g_n=\tilde{f}_n-\alpha_{1,n} e_1+\sqrt{\epsilon_n} e_1$ and it is clear that $g_n$ and $\tilde{f}_n$ enjoy the same aforementionned properties. Finally we set $h_n=-g_n+2\sqrt{\epsilon_n} e_1$. We claim that $\{g_n,h_n\}$ fulfill the assumptions a), b), c), d) on the probability space $([0,1],\mathcal{B}([0,1]),\lambda)$. Since $\tilde{f}_n$ and $e_1$ are centered we get that $g_n$ and $h_n$ are centered. Since $f_n$ is uniformly bounded as well as the sequences $\{\alpha_{1,n}\}_{n\ge 1}$ and $\{\epsilon_n\}_{n\ge 1}$ the same conclusion is valid for $\{g_n,h_n\}_{n\ge 1}$. Thus a) and b) hold. Besides, $g_n+h_n=2\sqrt{\epsilon_n} e_1$ tends to zero almost surely and d) is true. Finally, the orthogonality assumption c) also holds as illustrated by the next computations:
\begin{eqnarray*}
\int_0^1 g_n(x) h_n(x) dx &=& -\int_0^1 g_n^2(x) dx +2 \epsilon_n\\
&=& -2 \epsilon_n+2 \epsilon_n\\
&=&0.
\end{eqnarray*}

As a conclusion, the orthogonality assumption  is not sufficiently powerful to \textit{decompose} the almost sure convergence on the orthogonal projections. This short note explores the previous problem in the specific framework of Gaussian Wiener chaoses and homogeneous sums in independent random variables, which are well known to  display much more structure besides orthogonality. In this setting, we shall see that the previous question has a \textit{positive} answer.

Let us begin with one of our main results: We assume that all the random objects in this note are defined on a common probability space $\big( \Omega, \mathscr{F}, \mathbb{P} \big)$.

\begin{thm}\label{thm-Gauss}
Fix an integer $q\geq 2$  and  let $\{F_n, n\geq 1\}$ be a sequence in the sum of the first $q$ Gaussian  Wiener chaoses $($associated to a given isonormal Gaussian process $W$$)$ such that as $n\to+\infty$,
$F_n$ converges almost surely to  some random variable, denoted by $F$. Then,   $F$   belongs to the  sum of the first $q$ Gaussian  Wiener chaoses. Moreover, as $n\to+\infty$,
\begin{align}\label{ASG}  
  \text{$J_p(F_n)$ converges almost surely to $J_p(F)$}  
\end{align}
    for each $p\in\{0,\cdots,q\}$. Here $J_p(\bullet)$ denotes the projection operator onto the $p$th Gaussian Wiener chaos, see Section \ref{sec-Gauss} for more details. 

\end{thm}

The concept of  chaos that we consider dates back to N. Wiener's 1938 paper \cite{Wiener38}, in  which Wiener first introduced the notion of \emph{multiple Wiener integral} calling it polynomial chaos. It was later modified in K. It\^o's work \cite{Ito51, Ito56} so that  the multiple Wiener integrals of different orders are orthogonal to each other. The  multiple Wiener integrals   have proved to be a very useful tool
in the investigation of   the asymptotic behavior of partial sums of dependent random
variables; see P. Major's book \cite{Major-book}. Recently, the  multiple Wiener integrals that we call Wiener chaos here have   gained  considerable attention, as one can see along the expanding research line passing from D. Nualart and G. Peccati's fourth moment theorem \cite{FMT}
to the birth of the so-called Malliavin-Stein approch \cite{NP09MS}; see  the book \cite{bluebook} and \cite[Chapter 1]{GZthesis}  for more details.

   The results in \cite{Major-book, bluebook, FMT} mostly concern the convergence in distribution, and related ideas have been used to investigate the total-variation convergence on Wiener chaoses, see  for example \cite{NourdinPoly13}. Our Theorem \ref{thm-Gauss}, to the best of our knowledge, is the very first instance that focuses on the almost sure convergence on Wiener chaoses.  Note that the almost sure convergence in \eqref{ASG} occurs along the whole sequence, {\it not just} along some subsequence. In fact,  the almost sure convergence   along some subsequence is quite straightforward:  As a consequence of \emph{hypercontractivity property} of the \emph{Ornstein-Uhlenbeck semigroup}, the almost sure convergence of $F_n$ to $F$ implies the $L^2(\P)$ convergence of $F_n$ to $F$, so that $F$ belongs to the sum of the first $q$ Gaussian Wiener chaoses;   at the same time, we have $J_p(F_n)$ converges in $L^2(\P)$ to $J_p(F)$ because of the orthogonality property, hence the desired almost sure convergence holds true along some subsequence; see Section \ref{sec-Gauss} for more explanation.
  
   Let us first sketch our strategy of showing \eqref{ASG}, which motivates us to investigate the almost sure convergence on homogeneous polynomials in independent random variables.  By assumption, $F_n = \mathfrak{f}_n(W)\to  F=\mathfrak{f}(W)$ almost surely, where $\mathfrak{f}_n, \mathfrak{f}$ are deterministic representations of $F_n$ and $F$ respectively. Then it also holds true that $\mathfrak{f}_n(W^t)\to \mathfrak{f}(W^t)$ almost surely for each $t>0$, where 
   \begin{align}\label{Wt}
   W^t : = e^{-t}W + \sqrt{1 - e^{-2t}} \wh{W} \quad\text{with $\wh{W}$ an independent copy of $W$,}
   \end{align}
   is an isonormal Gaussian process that is exchangeable with $W$.  That is, $\big( W^t, W \big)$ has the same law as $\big( W, W^t \big)$ and in particular, $W^t$ has the same law as $W$. Conditioning on $W$, $\mathfrak{f}_n(W^t)$ belongs to the sum of the first $q$ Gaussian Wiener chaoses associated to $\wh{W}$ and converges to  $\mathfrak{f}(W^t)$ almost surely, therefore we deduce from the hypercontractivity that  
   \begin{center}
   conditioning on $W$, $\mathfrak{f}_n(W^t)$ converges in $L^2\big(\Omega, \sigma\{ \wh{W} \}, \P \big)$ to $\mathfrak{f}(W^t)$. 
   \end{center} 
   It follows immediately by taking expectation with respect to $\wh{W}$ that almost surely,
   \[
   \E\Big[ \mathfrak{f}_n(W^t) \vert \sigma\{W\} \Big] \to  \E\Big[ \mathfrak{f}(W^t) \vert \sigma\{W\} \Big] \,.
   \]
By Mehler's formula \eqref{Mehler}, we rewrite the above convergence as follow: $P_t F_n \to P_t F$, where   $(P_t, t\geq 0)$ is the Ornstein-Uhlenbeck semigroup associated to $W$.  By the definition of this semigroup, one has 
\begin{align}\label{OUASG}
P_t F_n = \sum_{k=0}^q e^{-kt} J_k(F_n) \quad\text{converges almost surely to } \quad P_t F = \sum_{k=0}^q e^{-kt} J_k(F)
\end{align}
   for each $t > 0$.  With many enough $t > 0$, we can deduce   \eqref{ASG} from \eqref{OUASG}.
   
   Now let us summarize the main idea of the proof: We introduced extra randomness $\wh{W}$ to construct an identical copy $W^t$ of $W$, then we carried out the conditioning argument that ``separates'' Wiener chaoses of different orders (see \eqref{OUASG}).  
   
   The above construction \eqref{Wt} of exchangeable pairs has been used in \cite{NZ17} to provide an elementary proof of the quantitative fourth moment theorem in \cite{NP09MS}, whose ideas have been further extended in the Ph.D thesis \cite{GZthesis}.   We remark here that the use of exchangeable pairs in aforementioned references is for proving the distributional convergence  within the framework of Stein's method of probabilistic approximation \cite{ICM14, ChenPoly, Stein86}; while we use  exchangeable\footnote{In fact, we just need the identical pairs to proceed our arguments in this note.} pairs to investigate the almost sure convergence, which is of some independent interest.

As announced, following the above strategy (but using different constructions of exchangeable pairs), we are able to provide   two more   results concerning  homogeneous sums.  Let us introduce a few notation first.

\medskip

\noindent{\bf Notation.}  Throughout this note, $\N: =\{1, 2, \ldots \}$,  $\mathbb{X} = \{ X_i, i\in\N \}$ denotes a sequence of centered independent real random variables and  $ \ell^2_0(\N)^{\odot p}$, with $p\geq 1$, stands for the set of kernels $f\in  \ell^2(\N)^{\otimes p}$ that is symmetric in its arguments and vanishes on its diagonals, that is, $f(i_1, \ldots, i_p) =0$ whenever $i_j = i_k$ for some $j\neq k$.  Given $f\in  \ell^2_0(\N)^{\odot p}$, we define the homogeneous sums (or homogeneous polynomials) with order $p$, based on the kernel $f$, by setting
\begin{align}\label{QPF}
Q_p\big(f; \mathbb{X} \big) = \sum_{i_1, \ldots, i_p\in\N} f(i_1, \ldots, i_p) X_{i_1} \cdots X_{i_p} \,,
\end{align}
whenever the above sum is well defined, for example when $f$ has finite support, or when $X_i$ has unit variance for each $i\in\N$.

\begin{obs}\label{rem12} {\rm  If $\mathbb{X}$ is a sequence of symmetric Rademacher random variables\footnote{That is, $X_i$ takes values in $\{-1, +1\}$ with equal probability for each $i\in\N$.}, the random variable defined in \eqref{QPF} is an element in    the so-called Rademacher chaos of order $p$; if $\mathbb{X}$ is a sequence of centered standard Poisson random variables\footnote{That is, $X_i = P_i - 1$ with $\{ P_i, i\in\N\}$ a sequence of i.i.d. standard Poisson random variables.}, the random variable $Q_p\big(f; \mathbb{X} \big)$ in \eqref{QPF} is a particular example of the Poisson Wiener chaos of order $p$; see Section 1.3 in \cite{GZthesis}.

}
\end{obs}

We are now in a position to state our results.

\begin{thm} \label{HP-bdd} Fix an integer $q\geq 2$ and let  $\mathbb{X}$ be a sequence  of centered independent real random variables such that $\E[ X_i^2] =1$ for each $i\in\N$ and $ \sup\{\E[ \vert X_j\vert^{2+\delta}] : j\in\N   \} < \infty$ for some $\delta>0$.   Assume $\big(F_n, n\geq 1\big)$ is a sequence of random variables satisfying the following conditions:
\begin{enumerate}
 \item[\rm (i)] For each $n\geq 1$, $F_n$ is of the form
 \begin{align}\label{Fnp}
 F_n = \E[F_n] + \sum_{p=1}^q Q_p\big(f_{p,n}; \mathbb{X} \big) \quad\text{with $f_{p,n}\in\ell^2_0(\N)^{\odot p}$ for each $p\in\{1,\ldots, q\}$.}
 \end{align}

\item[\rm (ii)] $F_n$ converges almost surely to  the random variable  $F$ of the following form:
 \begin{align}\label{Fpinf}
F =  \E[ F] + \sum_{p=1}^q Q_p(f_{p}; \mathbb{X}) \quad\text{with $f_{p}\in\ell^2_0(\N)^{\odot p}$ for each $p\in\{1,\ldots, q\}$.}
\end{align}

\end{enumerate}
Then, as $n\to+\infty$,  $\E[F_n]$ converges to $\E[F]$; and for each $p\in\{1, \ldots, q\}$,
 $Q_p\big(f_{p,n}; \mathbb{X} \big) $ converges almost surely to  $Q_p\big(f_{p}; \mathbb{X} \big) $.  

\end{thm}

Note that Theorem \ref{HP-bdd} covers the examples from Remark \ref{rem12}: In particular when  $\mathbb{X}$ is a sequence given therein, we do not need to assume that $F$ belongs to the sum of the first $q$  chaos, because this comes as a consequence of the hypercontractivity property (Lemma \ref{hyper}).

\medskip

  In the following result, we remove the $(2+\delta)$-moment assumption  with the price of  imposing finite-support assumption. It is worth pointing out that we only require $X_k$'s have first moments.

\begin{thm}[\bf Unbounded case]\label{HP-ubdd}   Fix an integer $q \geq 2$ and consider a sequence of centered independent {\rm (non-deterministic)} random variables $\mathbb{X} = \{X_k, k\in\N \}$. Let $F_n$ be a sequence  of random variables given by 
\[
F_n = \E[F_n] +  \sum_{p=1}^q Q_p\big(f_{p,n}; \mathbb{X}\big) \quad\text{with $f_{p,n}\in\ell^2_0(\N)^{\odot p}$ for each $p\in\{1,\ldots, q\}$.}
\]
 We assume that $d > q$ is an integer such that the support of $f_{p,n}$ is contained in $\{1, \ldots, d\}^p$ for any $p=1,\ldots, q$, $n\geq 1$. If  $F_n$ converges almost surely to $F$  with $F$ given by
 \[
 F =   \E[F] +  \sum_{p=1}^q Q_p\big(f_{p}; \mathbb{X}\big) \quad\text{with $f_{p}\in\ell^2_0(\N)^{\odot p}$ for each $p\in\{1,\ldots, q\}$.}
\]
  Then, as $n\to\infty$, $\E[F_n]\to \E[F]$; and $\forall p\in\{1,\ldots, q\}$,  $Q_p\big(f_{p,n}; \mathbb{X}\big)$ converges almost surely to $Q_p\big(f_{p}; \mathbb{X}\big)$.

\end{thm}

This paper  naturally leads to the following open questions.

\noindent{\bf Open question 1.}  Let $\mathbb{Y}$ be a sequence of independent (nonsymmetric) Rademacher variables, that is, $\P\big( Y_k = 1) = p_k = 1 - \P\big( Y_k = -1)\in(0,1)$ for each $k\in\N$. Define $X_k = (Y_k +1 - 2p_k)/(2\sqrt{p_k - p_k^2} )$ for each $k\in\N$. Assume $F_n$ and $F$ have the form \eqref{Fnp} and \eqref{Fpinf} respectively such that $F_n$ converges almost surely to $F$. Does the almost surely convergence in Theorem \ref{HP-bdd} hold true?

\medskip

\noindent{\bf Open question 2.}    Fix an integer $q\geq 2$ and let $(F_n, n\in\N)$ be a sequence of random variables that belong to the sum of the first $q$ Poisson Wiener chaoses\footnote{Say, these Poisson Wiener chaoses are defined based  a Poisson random measure over a $\sigma$-finite measure space; see the book \cite{LP-book} for more details.} such that $F_n$ converges almost sure to some $F$ in the sum of the first  $q$ Poisson Wiener chaoses.
\begin{center}
{\it  Does  $J_p(F_n)$ converge almost surely to $J_p(F)$, as $n\to+\infty$?}
 \end{center} 
 Here   $J_p(\bullet)$ denotes the projection operator onto the $p$th  Poisson Wiener chaos, and one can refer to the book \cite{LP-book} for any unexplained term. 
  
\medskip

The rest of this article is organized as follows: Section \ref{sec-Gauss} is devoted to a detailed proof of Theorem \ref{thm-Gauss} and Section \ref{sec-HP} consists of two subsections that deal with Theorem \ref{HP-bdd} and Theorem \ref{HP-ubdd} respectively.

\section{Almost sure convergence on the Gaussian Wiener chaoses}\label{sec-Gauss}

   Let $W:= \big\{ W(h) : h\in\H \big\}$ be an isonormal Gaussian process over a real separable Hilbert space $\H$, \emph{i.e.} $W$ is a centered Gaussian family such that $\E[ W(h) W(g) ] = \langle h, g \rangle_\H$ for any $h, g\in\H$. The resulting $L^2(\Omega, \sigma\{W\}, \P)$ can be decomposed in an {\it orthogonal} manner as a direct sum of Gaussian Wiener chaoses:
   $
   L^2(\Omega, \sigma\{W\}, \P) = \bigoplus_{p=0}^\infty \C^W_p 
   $\,
 where $\C^W_0 = \R$ and for $p\geq 1$, the $p$th Gaussian Wiener chaos admits a complete orthonormal system given by
 \begin{align}\label{OBCP}
\left\{  \prod_{i\in\N} \frac{H_{a_i}(W(e_i) )}{\sqrt{a_i!}} \,: \, a_i\in\N\cup\{0\} \,\, \text{such that} \,\, \sum_{i\in\N}a_i = p   \right\} \,.
 \end{align}
 In this note, $H_p(x) := (-1)^p \exp(x^2/2) \frac{d^p}{dx^p}\exp(-x^2/2) $     is the $p$th Hermite polynomial\footnote{The first few Hermite polynomials are given by $H_0(x) =1$, $H_1(x) = x$, $H_2(x) = x^2 - 1$ and $H_{p+1}(x) = xH_p(x) - pH_{p-1}(x)$ for any $p\in\N$. } and $\{ e_i : i\in\N\}$ stands for an orthonormal basis of $\H$. 
 
 Recall that  $J_p(\bullet)$ denotes the projection operator onto $\C^W_p$, and we can define the Ornstein-Uhlenbeck semigroup $(P_t, t\geq 0)$ by setting
 $
 P_t  F =  \E[ F] + \sum_{p=1}^\infty e^{-pt} J_p(F)$,
 $\forall F\in L^2(\Omega, \sigma\{W\}, \P)$.
 
 It also has the following nice representation that is of central importance to our approach: given $F\in L^2(\Omega, \sigma\{W\}, \P)$, we first have $F = \mathfrak{f}(W)$ for some deterministic representation $\mathfrak{f}: \R^\H\to \R$, then   Mehler's formula reads as follows:
    \begin{align}\label{Mehler}
    P_tF= \E\big[ \mathfrak{f}(W^t) \vert \sigma\{ W\} \big]
    \end{align}
 where $W^t$ is another isonormal Gaussian process defined in \eqref{Wt}.   As a consequence of the hypercontractivity  property of the Ornstein-Uhlenbeck semigroup, we have the equivalence of all $L^r(\P)$-norm $(1< r < +\infty)$ on a fixed Gaussian Wiener chaos.    For any unexplained term, one can refer to  \cite{bluebook, Nualart06}.   
 
 \medskip
 
 Now we are ready to present the proof of Theorem \ref{thm-Gauss}.

\begin{proof}[Proof of Theorem \ref{thm-Gauss}]    As a consequence of the \emph{hypercontractivity} (see \emph{e.g.} \cite[Section 1.4.3]{Nualart06}),  $\{F_n, n\geq 1\}$ is bounded in $L^m(\P)$ for any $m>1$. Therefore,  $\E[ \vert F\vert^m] < +\infty$ for any $m >1$ and 
\[
\E\big[ (F_n - F)^2 \big] = \E[ F_n^2] +  \E[ F^2] - 2 \E[ F_n F] \xrightarrow{n\to+\infty} \E[ F^2] +  \E[ F^2] - 2 \E[ F^2] =0 \, ,
\]
where the limit follows from the almost sure convergence of $F_nF$ and its uniform integrability.  So we can conclude that $F$ belongs to the sum of the first $q$ Gaussian Wiener chaoses.

\noindent{\bf Observation:} Given $G = \mathfrak{g}(W)\in L^2(\P)$ for some deterministic representation $\mathfrak{g}:\R^\H\to \R$, then in view of \eqref{OBCP},   $G$ belongs to the sum of the first $q$ Gaussian Wiener chaoses (associated to $W$) if and only if $G$ is a polynomial in i.i.d. Gaussians $\{ W(e_i)\,:\, i\in\N \}$ with degree $\leq q$. It follows from this equivalence that conditioning on $W$, $\mathfrak{g}(W^t)$ belongs to the sum of the first $q$ Gaussian Wiener chaoses (associated to $\wh{W}$), which can be seen from the simple formula $H_n(ax+by) = \sum_{k=0}^n {n\choose k} a^kb^{n-k}H_k(x) H_{n-k}(y)$ with $a,b\in\R$ such that $a^2 + b^2 =1$.

\medskip

With the above observation and assuming $F_n = \mathfrak{f}_n(W)$, $F = \mathfrak{f}(W)$ for some deterministic $\mathfrak{f}_n$ and $\mathfrak{f}$,  it holds that $\mathfrak{f}_n(W^t)$ converges almost surely to $\mathfrak{f}(W^t)$. Therefore, conditioning on $W$,  $\mathfrak{f}_n(W^t)$ belongs to  the sum of the first $q$ Gaussian Wiener chaoses (associated to $\wh{W}$) for each $n\in\N$ and converges almost surely to $\mathfrak{f}(W^t)$. Thus, due to the hypercontractivity, we have conditioning on $W$, the expectation of $\mathfrak{f}_n(W^t)$ with respect to $\wh{W}$ converges to that of $\mathfrak{f}(W^t)$ with respect to $\wh{W}$. That is, almost surely, 
$
P_tF_n = \E\big[ \mathfrak{f}_n(W^t)\, \vert \sigma\{W\} \big] \to  \E\big[ \mathfrak{f}(W^t)\, \vert \sigma\{W\} \big]  = P_tF
$, as $n\to+\infty$. Note that the two equalities in the above display follow from Mehler's formula \eqref{Mehler}. Hence by definition, we have 
\[
\E[F_n] + \sum_{k=1}^q e^{-kt} J_k(F_n) \longrightarrow \E[F] + \sum_{k=1}^q e^{-kt} J_k(F) \,
\]
almost surely, as $n\to+\infty$.  Applying the above argument for sufficiently many  $t > 0$ gives the desired result. \qedhere  

\end{proof}

As an interesting corollary, we show that, for sequences of random variables lying in a fixed finite sum of Gaussian Wiener chaoses, the almost sure convergence is robust to the application of the standard Malliavin calculus operators $L$ and $\Gamma$, where the {\it Ornstein-Uhlenbeck} generator $L$ is defined formally by $L = \sum_{p\geq 1} -p J_p$ and the \emph{carr\'e-du-champ} operator is defined by  $\Gamma[F,G] = \frac{1}{2}\big( L[FG] - FL[G] - GL[F] \big)$ whenever the expressions make sense. For example, when $F, G$ live in a fixed finite sum of Gaussian Wiener chaoses, $L[F]$ and $\Gamma[F,G]$ are well defined. 

More precisely, we have the following result.

\begin{cor}\label{corr0}
For  fixed integers $p, q\in\N$, we consider a sequence $F_n \in \bigoplus_{k\le p} \mathbb{C}_k^{W}$ which converges almost surely towards $F$, as $n\to+\infty$ Then, we have 
$
L[F_n] \xrightarrow[n\to\infty]{\text{a.s.}}  L[F] \,.$

Consider another sequence $G_m \in \bigoplus_{k\le q} \mathbb{C}_k^{W}$ that  converges almost surely towards $G$, as $m\to+\infty$. Then, we have 
$
\Gamma[F_n, G_n] \xrightarrow[n\to\infty]{\text{a.s.}}  \Gamma[F, G]$.

\end{cor}
\begin{proof}
First of all, Theorem \ref{thm-Gauss} tells us that $J_k(F_n)$ converges towards $J_k(F)$ almost surely for every $k\in\{ 0, \ldots, p\}$. Thus,   we get
\[
L[F_n]= \sum_{k=0}^p -k J_k(F_n) \xrightarrow[n\to\infty]{\text{a.s.}}  L [F].
\]This achieves the proof of the first assertion and in the same way, $ L[G_n] \xrightarrow[n\to\infty]{\text{a.s.}}  L[G] \,.$

Second assertion follows easily from the first one: Clearly $F_nG_n$ lives in the sum of the first $(p+q)$ Gaussian Wiener chaos in view of the multiplication formula (see \emph{e.g.} \cite{Nualart06}) and  $F_nG_n$ converges to $FG$ almost surely, thus $L[F_nG_n]$ converges almost surely to $L[FG]$, implying $
\Gamma[F_n, G_n] \xrightarrow[n\to\infty]{\text{a.s.}}  \Gamma[F, G]$. This completes our proof. \qedhere

\end{proof}

\begin{obs} {\rm As we have pointed out, Theorem \ref{HP-bdd} covers the case where $F_n$ belongs to a fixed sum of  Rademacher chaoses (in the symmetric setting), so a similar result to the above Corollary \ref{corr0} can be formulated.

}

\end{obs}
      
  \section{Almost sure convergence on homogeneous sums}\label{sec-HP}
  
  This section is divided into two parts, presenting the proof of Theorems \ref{HP-bdd},  \ref{HP-ubdd} respectively.  Having our general strategy in mind, we will first present our use of auxiliary randomness   both in Section \ref{bddsec} and in Section \ref{ubddsec}.

  \subsection{The  case of bounded $(2+\delta)$ moments}\label{bddsec}   Recall that $\mathbb{X} = ( X_i, i\in\N)$ is a sequence of centered independent random variables. Let $\wh{\mathbb{X}} = \big( \wh{X}_i, i\in\N \big)$ is an independent copy of $\mathbb{X}$ and let $\Theta = \big( \theta_i, i\in\N \big)$ be a sequence of i.i.d. standard exponential random variables such that $\mathbb{X},\wh{\mathbb{X}}$ and $\Theta$ are independent. Given any $t > 0$, we define a new sequence $\mathbb{X}^t= ( X^t_i, i\in\N)$ by setting
  $
  X^t_i = X_i\cdot\1_{\{\theta_i \geq t \}} + \wh{X}_i\cdot \1_{\{\theta_i < t \}}$  for each $i\in\N$.     It is routine to verify that $\mathbb{X}^t$ has the same law as   $\mathbb{X}$: Indeed,  given any $y\in\R, i\in\N$, we have 
$
\P\big( X_i^t \leq y \big) = \P\big( X_i \leq y ,  \theta_i \geq t \big) +  \P\big( \wh{X}_i \leq y ,  \theta_i < t \big) = \P\big( X_i \leq y \big) 
$
  and moreover  it is clear that $\mathbb{X}^t$ is a sequence of independent random variables, thus $\mathbb{X}^t$ has the same law  as   $\mathbb{X}$.

   Another ingredient for our proof is the following {\it hypercontractivity property}.

   \begin{lemma}\label{hyper} Let $\Xi=(\xi_i, i\in\N)$ be a sequence of real centered independent random variables such that $\E[ \xi_i^2] =1$ for each $i\in\N$ and  $M:=\sup\big\{\E[  \vert \xi_j\vert^{2+\delta}  ] \,:\, j\in\N   \big\}$ is finite for some $\delta > 0$.   Given any $f\in\ell^2_0(\N)^{\odot d}$, one has
   \begin{align}\label{onehas}
   \big\| Q_d(f ; \Xi ) \big\| _{L^{2+\delta}(\P)} \leq  \left(  2\sqrt{1+\delta} \cdot M^{1/(2+\delta)}   \right)^d  \big\| Q_d(f ; \Xi ) \big\| _{L^2(\P)} \,.
   \end{align}
As a consequence: Given a fixed integer   $q\geq 2$, 
$
F_n := \E[F_n] + \sum_{p=1}^q Q_p\big(f_{p,n};\Xi \big)$,
with $f_{p,n}\in\ell^2_0(\N)^{\odot p}$ for each $p\in\{1,\ldots, q\}$, if $\{F_n\}_{n\geq 1}$ is tight, 
then $\sup\big\{ \E[  \vert F_n\vert^{2+\delta} ] : n\in\N \big\}< +\infty$. 
   
      \end{lemma}

 \begin{proof} Our lemma follows easily from Propositions 3.11, 3.12 and 3.16 in \cite{MOO10}:  let us first truncate the kernel $f_n  = f \cdot\1_{\{1, \ldots, n\}^d}$ with any $n\geq d$, then $Q_d(f ; \Xi )$ is simply a multilinear polynomial in $\xi_1, \ldots, \xi_n$ so that the results in   \cite{MOO10} imply 
 $
  \big\| Q_d(f_n ; \Xi ) \big\| _{L^{2+\delta}(\P)} \leq \left(  2\sqrt{1+\delta} \cdot M^{1/(2+\delta)}   \right)^d   \big\| Q_d(f_n ; \Xi ) \big\| _{L^2(\P)} $,
 thus by passing $n$ to infinity, we get \eqref{onehas}; and by Minkowski's inequality, we get $\| F_n \| _{L^{2+\delta}(\P)} \leq \kappa\cdot  \| F_n  \| _{L^2(\P)}$ with $\kappa:= \left(  2\sqrt{1+\delta} \cdot M^{1/(2+\delta)}   \right)^q (q+1)$.  Note that
 \begin{align*}
 \E\big[  F_n^2 ] \leq \E\big[  F_n^2 \cdot\1_{\{  F_n^2 > \E[ F_n^2]/2  \}} \big] + \frac{1}{2} \E[ F_n^2] \leq  \frac{1}{2} \E[ F_n^2]+ \| F_n \| _{L^{2+\delta}(\P)}^{2} \cdot \P\big( F_n^2 > \E[ F_n^2]/2 \big)^{\delta/(2+\delta)}
 \end{align*}
 where the last inequality follows from the H\"older inequality.  Therefore, 
 \[
\P\big( F_n^2 > \E[ F_n^2]/2 \big) \geq \left(1/2\right)^{(2+\delta)/\delta} \left( \| F_n \| _{L^2(\P)}  \big/ \| F_n \| _{L^{2+\delta}(\P)}  \right)^{2(2+\delta)/\delta}  \geq  (2\kappa^2)^{-(2+\delta)/\delta} \,\,~\text{for each $n\in\N$}
 \]
 while due to tightness, one can find $K > 0$ large enough such that $\P\big( F_n^2 > K \big) <  (2\kappa^2)^{-(2+\delta)/\delta}$ for each $n\in\N$, implying that $ \P\big(  F_n^2 > \E[ F_n^2]/2 \big)  > \P\big( F_n^2 > K \big)$. This gives us $\E[ F_n^2] \leq 2K $, $\forall n\in\N$. Hence, $\| F_n \| _{L^{2+\delta}(\P)} \leq \kappa \cdot \sqrt{2K} < +\infty$ for each $n\in\N$.                    \end{proof}

\begin{proof}[Proof of Theorem \ref{HP-bdd}]    Without losing any generality, we can assume that $F_n = \mathfrak{f}_n(\mathbb{X})$ and $F = \mathfrak{f}(\mathbb{X})$ for some deterministic mappings $\mathfrak{f}_n, \mathfrak{f}$ from $\R^\N$ to $\R$.  By our construction of $\mathbb{X}^t$, it follows that
$
 \mathfrak{f}_n(\mathbb{X}^t) =  \E[ F_n] + \sum_{p=1}^q Q_p\big(f_{p,n}; \mathbb{X}^t\big) \longrightarrow \mathfrak{f}(\mathbb{X}^t) =  \E[ F] + \sum_{p=1}^q Q_p\big(f_{p}; \mathbb{X}^t\big)$ almost surely.  Now fixing a generic realization of $\mathbb{X}$, $\mathfrak{f}_n(\mathbb{X}^t)$ and $\mathfrak{f}(\mathbb{X}^t)$ are polynomials (with bounded degrees) in 
 \[
\Xi:= \left\{  \wh{X}_1,  \frac{\1_{\{ \theta_1 < t \}} -(1- e^{-t})}{  \sqrt{ e^{-t}(1-e^{-t})  }  } ,  \wh{X}_2,  \frac{\1_{\{ \theta_2 < t \}} -(1- e^{-t})}{  \sqrt{ e^{-t}(1-e^{-t})  } } ,  \wh{X}_3,  \frac{\1_{\{ \theta_3 < t \}} -(1- e^{-t})}{ \sqrt{ e^{-t}(1-e^{-t})  }  } \,, \ldots \, \right\}
 \]
 and $\mathfrak{f}_n(\mathbb{X}^t)$ converges to $\mathfrak{f}(\mathbb{X}^t)$ almost surely  with respect to the randomness $\Theta, \wh{\mathbb{X}}$, as $n\to+\infty$.
 
Note that $\Xi$ defined above satisfies the assumptions in Lemma \ref{hyper}.  Thus, by fixing a generic realization of $\mathbb{X}$, $\mathfrak{f}_n(\mathbb{X}^t)$ is uniformly bounded in $L^{2+\delta}(\P)$ (with respect to the randomness $\Theta, \wh{\mathbb{X}}$).  Therefore,   we have 
$
 \E\big[ \mathfrak{f}_n( \mathbb{X}^t) \vert \mathbb{X} \big]   \longrightarrow  \E\big[ \mathfrak{f}( \mathbb{X}^t) \vert \mathbb{X} \big]  
$,
as $n\to+\infty$.  The above conditional expectations can be easily computed as follows: 
\begin{align*}
 \E\big[ \mathfrak{f}_n( \mathbb{X}^t) \vert \mathbb{X} \big]& = \E[F_n] + \sum_{p=1}^q  \E\big[      Q_p\big(f_{p,n}; \mathbb{X}^t\big)          \vert \mathbb{X} \big]  = \E[F_n] + \sum_{p=1}^q  \sum_{i_1, \ldots, i_p\in\N}f_{p,n}(i_1, \ldots, i_p)   \E\big[  X^t_{i_1}\cdots X^t_{i_p}     \vert \mathbb{X} \big] \\
&= \E[F_n] + \sum_{p=1}^q  \sum_{i_1, \ldots, i_p\in\N}f_{p,n}(i_1, \ldots, i_p)  e^{-pt} X_{i_1}\cdots X_{i_p} =  \E[ F_n] + \sum_{p=1}^q e^{-pt} Q_p\big(f_{p,n}; \mathbb{X}\big) \,;
 \end{align*}
 in the same way, we have $ \E\big[ \mathfrak{f}( \mathbb{X}^t) \vert \mathbb{X} \big] =   \E[ F] + \sum_{p=1}^q e^{-pt} Q_p\big(f_{p}; \mathbb{X}\big)$.   Hence the desired result follows from the same ending argument as in previous section.   \qedhere
 
\end{proof}

\subsection{The unbounded case}\label{ubddsec}   Recall from the statement of Theorem \ref{HP-ubdd} that  $\mathbb{X}$ is a sequence of centered independent random variables and the kernels $(f_{p,n}, n\geq 1)$ have their supports uniformly contained in $\{1, \ldots, d\}^p$, for each $p\in\{1, \ldots, q\}$. It follows that $F_n$ and $F$ only depend on the first $d$ coordinates $X_1, \ldots, X_d$.

Now given any $m \geq  d$, in what follows, we   construct random variables $X^{(I)}_1, \ldots, X^{(I)}_m$ such that $\big(X^{(I)}_1, \ldots, X^{(I)}_m\big)$ is equal in law to $(X_1, \ldots, X_m)$: Let $\wh{\mathbb{X}}$ be an independent copy of $\mathbb{X}$ and let $I:=I_m$ be a uniform random variable on $[m]:=\{1, \ldots, m\}$ such that $\mathbb{X}$, $\wh{\mathbb{X}}$ and $I$ are independent, then for each $i, k\in[m]$, we define 
$
X_k^{(i)} = X_k \cdot \1_{\{ i \neq k \}} + \wh{X}_k \cdot \1_{\{ i = k \}}$  and  $X_k^{(I)} = X_k \cdot \1_{\{ I \neq k \}} + \wh{X}_k \cdot \1_{\{ I = k \}}$.
                It is routine to verify that $\big(X^{(I)}_1, \ldots, X^{(I)}_m\big)$ and  $(X_1, \ldots, X_m)$   are exchangeable pairs, which  is well-known in the community of Stein's method, see for instance \cite{Stein86}.

 \begin{proof}[Proof of Theorem \ref{HP-ubdd}]  Since $F_n, F$ only depend on the first $d$ coordinates, we can write 
 $$Q_p\big(f_{p}; \{X_k, k\in[m] \} \big) = Q_p\big(f_{p};  \mathbb{X}\big) \quad\text{for any $m\geq d$.}
 $$
With this convention, we   define $F^{(i)}$ to be the following sum
 $
\E[F]+  \sum_{p=1}^q Q_p\big(f_{p}; \{X^{(i)}_k, k\in[m] \} \big) 
 $
 and similarly we define $F_n^{(I)}$, $F^{(I)}$ and $F_n^{(i)}$ for each $n\geq 1$. It follows   that $F_n^{(j)}$  converges almost surely to $F^{(j)}$ for each $j$, since $(X^{(j)}_k, k\in[m] )$ is clearly equal in law to $(X_k, k\in[m] )$.
 
 Now we fix $i\in[m]$, $F_n^{(i)}$ can be expressed as a sum of two parts: $F_n^{(i)} = \alpha_n(i) + \beta_n(i) \wh{X}_i$, where $\alpha_n(i)$ and  $\beta_n(i)$ are polynomials in $X_1, \ldots, X_{i-1}, X_{i+1}, \ldots, X_m$.  In the same way, we can rewrite  $F^{(i)} = \alpha(i) + \beta(i) \wh{X}_i$, where $\alpha(i)$ and  $\beta(i)$ are polynomials in $X_1, \ldots, X_{i-1}, X_{i+1}, \ldots, X_m$. 
(Note that $\beta_n(i) = \beta(i) =0$ for $i > d$.)  Therefore, conditioning on $\{X_1, \ldots, X_m\}$, $ \alpha_n(i)\to  \alpha(i)  $ and $\beta_n(i)\to \beta(i)$, as $n\to+\infty$. Thus,
  $
  \E\big[ F_n^{(i)} \vert \mathbb{X} \big] =  \alpha_n(i)\to  \alpha(i) =  \E\big[ F^{(i)} \vert \mathbb{X} \big] $.
 Hence,  almost surely
   \begin{align}\label{U1}
  \E\big[ F_n^{(I)} \vert \mathbb{X} \big] =  \frac{1}{m} \sum_{i=1}^m     \E\big[ F_n^{(i)} \vert \mathbb{X} \big]    \to       \frac{1}{m} \sum_{i=1}^m     \E\big[ F^{(i)} \vert \mathbb{X} \big]  =      \E\big[ F^{(I)} \vert \mathbb{X} \big] \,.
  \end{align}
 Let us now compute the above conditional expectations:  since
 \begin{align*}
&\quad  \E\Big[ Q_p\big(f_{p}; \{ X^{(I)}_k, k\in[m] \}  \big) \big\vert  \mathbb{X} \Big]  - Q_p\big(f_{p}; \mathbb{X} \big)  = \sum_{i=1}^m m^{-1}  \E\Big[ Q_p\big(f_{p}; \{ X^{(i)}_k, k\in[m] \}  \big)  - Q_p\big(f_{p}; \mathbb{X} \big)  \big\vert  \mathbb{X} \Big]   \\
& = -  \sum_{i=1}^m m^{-1}  \sum_{j_1, \ldots, j_p\leq d}  f_p(j_1, \ldots, j_p)  X_{j_1}\cdots  X_{j_p}   \cdot\1_{( i\in\{ j_1, \ldots,  j_p   \} )}  =  - p m^{-1}   Q_p\big(f_{p}; \mathbb{X} \big) \,,
 \end{align*}
we have $ \E\Big[ Q_p \big(f_{p}; \{ X^{(I)}_k, k\in[m] \}   \big)  \vert  \mathbb{X} \Big]  =  (1 - pm^{-1} ) Q_p(f_p; \mathbb{X} )$
 so that \eqref{U1} implies
 \[
\E[F_n]+ \sum_{p=1}^q (1 - p m^{-1} )  Q_p\big(f_{p,n}; \mathbb{X} \big) \xrightarrow[a.s.]{n\to\infty} \E[F]+ \sum_{p=1}^q (1 - p m^{-1} )  Q_p\big(f_{p}; \mathbb{X} \big) .
 \]
 Hence the desired result follows from    iterating the above process for many enough $m \geq d$.
 \end{proof}

\end{document}